\documentclass[12pt,reqno]{amsart}
\usepackage[a4paper]{geometry}
\usepackage{amsthm,hyperref,lmodern,mathrsfs,amssymb,amsmath,amsfonts}
\usepackage{enumerate}
\usepackage[latin1]{inputenc}
\usepackage{versions}

\usepackage{color}

\allowdisplaybreaks

\numberwithin{equation}{section}

\newtheorem{theo}{Theorem}[section]
\newtheorem{lemme}[theo]{Lemma}
\newtheorem{prop}[theo]{Proposition}

\def\cE{{\mathcal{E}\ \!\!}}

\def\N{{\mathbb{N}\ \!\!}}
\def\R{{\mathbb{R}}}
\def\L{{\mathcal{L}\ \!\!}} % sur les formes
\def\LL{{L \ \!\!}} % sur les fonctions

\def\D{{\mathcal{D}\ \!\!}}

\def\C{{\mathcal{C}\ \!\!}}

\def\Var{{\mathrm{{\rm Var}}}}

\def\Cov{{\mathrm{{\rm Cov}}}}

\def\and{{\mathrm{{\rm and}}}}

\def\ve{{\varepsilon\ \!\!}}

\begin{document}

\title[Intertwinings, Brascamp-Lieb inequalities and spectral estimates]{Intertwinings, second-order Brascamp-Lieb inequalities and spectral estimates}

\author{Michel~Bonnefont} \address[M.~Bonnefont]{UMR CNRS 5251, Institut de Math\'ematiques de Bordeaux, Universit\'e Bordeaux 1, France} \thanks{MB is partially supported by the French ANR-12-BS01-0013-02 HAB project} \email{\url{mailto:michel.bonnefont(at)math.u-bordeaux1.fr}} \urladdr{\url{http://www.math.u-bordeaux1.fr/~mibonnef/}}

\author{Ald\'eric~Joulin} \address[A.~Joulin]{UMR CNRS 5219, Institut de Math\'ematiques de Toulouse, Universit\'e de Toulouse, France} \thanks{AJ is partially supported by the French ANR-12-BS01-0019 STAB project} \email{\url{mailto:ajoulin(at)insa-toulouse.fr}} \urladdr{\url{http://perso.math.univ-toulouse.fr/joulin/}}



\maketitle

\begin{abstract} We explore the consequences of the so-called intertwinings between gradients and Markov diffusion operators on $\R^d$ in terms of second-order Brascamp-Lieb inequalities for log-concave distributions and beyond, extending our inequalities established in a previous paper. As a result, we derive some convenient lower bounds on the $(d+1)^{th}$ positive eigenvalue depending on the spectral gap of the dual Markov diffusion operator given by the intertwining. To see the relevance of our approach, we apply our spectral results in the case of perturbed product measures, freeing us from Helffer's classical method based on uniform spectral estimates for the one-dimensional conditional distributions.
\end{abstract}

\section{Introduction}
\label{sect:intro}
Following our previous study \cite{ABJ}, the purpose of these notes is to further explore the consequences in terms of spectral-type functional inequalities of the so-called intertwinings between gradients and Markov diffusion operators of the form
$$
L f = \Delta f - (\nabla V )^T \, \nabla f.
$$
Here $V$ is some nice potential on $\R^d$ such that the measure $\mu$ with Lebesgue density proportional to $e^{-V}$ is the unique invariant probability measure. Actually, the intertwining approach revealed to be a powerful tool to establish Poincar\'e and Brascamp-Lieb type inequalities for these operators, giving some important informations on its spectrum and more precisely on its spectral gap $\lambda_1$. Recall that the principle of the intertwining is to write the gradient of a given diffusion operator as a matrix operator of Schr\"odinger type acting on the gradient, and then to exploit the specific properties of this alternative operator. In \cite{ABJ} our idea was to consider in the intertwining a weighted gradient instead of the classical Euclidean gradient, the weight being given by the multiplication by an invertible square matrix $A$ depending on the space variable. As a result, the presence of this weight enabled us to address Brascamp-Lieb type inequalities for many different Markovian dynamics including the log-concave case, i.e. $V$ is convex, and beyond, each situation of interest corresponding more or less to a convenient choice of weight $A$. \smallskip

In the recent years, there have been various extensions of the Brascamp-Lieb inequality taking different forms, which can be established by several methods in connection with optimal transport and celebrated conjectures such as the (B) and KLS conjectures. See for instance the following non-exhaustive list of articles \cite{helffer,CFM,harge,bob_ledoux2,klartag,bj,nguyen,milman_koles,bgg} and mainly the recent paper of Cordero-Erausquin \cite{cordero} in which, using mass transportation techniques, he established among other things the following improved Brascamp-Lieb inequality: if the Hessian matrix $\nabla^2 V$ is positive-definite, then for every $\mu$-centered smooth function $f$ orthogonal to the coordinate functions in $L^2 (\mu)$, then
$$
\Var_\mu (f) \leq \int_{\R^d} (\nabla f)^T \, \left( \nabla^2 V + \lambda_1 \, I \right) ^{-1} \, \nabla f \, d\mu ,
$$
where $\Var_\mu (f)$ denotes the variance of $f$ with respect to $\mu$ and $I$ is the identity matrix. Applied to the uniform log-concave case, i.e. $V$ is uniformly convex, the quantity $\rho (\nabla^2 V)$ denoting the smallest eigenvalue of $\nabla^2 V$, such an inequality entails immediately the following lower bound on the $(d+1)^{th}$ positive eigenvalue,
$$
\lambda_{d+1} \geq \lambda_1 + \inf \, \rho (\nabla^2 V),
$$
a relevant estimate since the multiplicity of the spectral gap $\lambda_1$ is at most the dimension $d$, as noticed by Barthe and Cordero-Erausquin in \cite{barthe-cordero}, following an argument of Klartag \cite{klartag}. We use the terminology ``second-order Brascamp-Lieb inequality" to qualify Cordero-Erausquin's inequality since it is related to a higher order eigenvalue than the spectral gap $\lambda_1$. Although optimal in the standard Gaussian case, there is still room for extension of these results, in particular in view to relax the uniform convexity assumption on the potential $V$. \smallskip

Actually, Cordero-Erausquin's paper, which is reminiscent of a previous work with Barthe \cite{barthe-cordero}, is the starting point of the present paper. Indeed, we intend to make a further step in this direction by addressing these issues in more general situations and to investigate the consequences of these inequalities for higher order eigenvalues. Our approach, different from Cordero-Erausquin's one but more comparable to that used in \cite{barthe-cordero}, is based on the $L^2$ method of H\"ormander \cite{hormander} and combined with the intertwining technique, yields a family - parametrized by the matrix $A$ - of new second-order Brascamp-Lieb inequalities. As a result, we get some lower bounds on $\lambda_{d+1}$ for more general functions $V$. Certainly, estimating higher order eigenvalues for these operators or more generally for weighted Laplacians on Riemannian manifolds has a long story and can be treated classically by heat kernel upper bounds \cite{li_yau} and more recently by the so-called super Poincar\'e inequalities \cite{wang} or by comparison techniques between model spaces \cite{milman}. However our estimates on $\lambda_{d+1}$ are not really comparable to these ones since our objectives are somewhat different: instead of studying the potential optimality of the estimates with respect to the parameters of interest of the problem, we rather want to obtain some simple criteria on the potential $V$, allowing us to consider some interesting situations which do not enter into the previous framework, ensuring such convenient lower bounds to hold. \smallskip

Let us briefly describe the content of the paper. In Section \ref{sect:prelim}, we recall some basic material on Markov diffusion operators together with the underlying spectral quantities of interest. In the spirit of our previous work \cite{ABJ}, we introduce the so-called intertwinings between gradients and Markov diffusion operators. Section \ref{sect:BL} is then devoted to Brascamp-Lieb type inequalities of the first and second orders, with a special emphasis on the consequences in terms of spectral quantities, corresponding to our main results: we obtain some new lower bounds on the higher order eigenvalue $\lambda_{d+1}$ which depend on the spectral gap of the dual Markov diffusion operator given by the intertwining. Finally, we apply in Section \ref{sect:ex} our spectral results in the case of perturbed product measures. In particular our approach reveals to be an alternative to Helffer's classical method based on uniform spectral estimates for the one-dimensional conditional distributions.

\section{Preliminaires}
\label{sect:prelim}
\subsection{Basic material }
Let $\C ^\infty (\R ^d , \R)$ be the space of infinitely differentiable real-valued functions on the Euclidean space $(\R^d , \vert \cdot \vert )$, $d\geq 2$, and let $\C _0 ^\infty (\R ^d, \R)$ be the subspace of $\C ^\infty (\R ^d , \R)$ of compactly supported functions. Denote $\Vert \cdot \Vert _\infty$ the essential supremum norm with respect to the Lebesgue measure. We consider the Markov diffusion operator defined on $\C ^\infty (\R ^d , \R)$ by
$$
\LL f := \Delta f - (\nabla V) ^T \, \nabla f ,
$$
where $V$ is a smooth potential on $\R^d$ whose Hessian matrix $\nabla ^2 V$ is, with respect to the space variable, uniformly bounded from below (in the sense of symmetric matrices). Above $\Delta$ and $\nabla$ stand respectively for the Euclidean Laplacian and gradient and the symbol $^T$ means the transpose of a column vector (or a matrix). Let $\Gamma$ be the \textit{carr\'e du champ} operator which is the bilinear symmetric form defined on $\C ^\infty (\R ^d , \R) \times \C ^\infty (\R ^d , \R)$ by
$$
\Gamma (f,g) := \frac{1}{2} \, \left( \LL (fg) - f \, \LL g - g \, \LL f \right) = (\nabla f) ^T \, \nabla g .
$$
If $e^{-V}$ is Lebesgue integrable on $\R^d$, a condition which will be assumed throughout the whole paper, then we denote $ \mu$ the probability measure with Lebesgue density proportional to $e^{-V}$ on $\R^d$. The operator $\LL$, which is symmetric with respect to $\mu$, that is, for every $f,g \in \C _0 ^\infty (\R ^d , \R)$,
$$
\cE _\mu (f,g)  := \int_{\R ^d} \Gamma (f,g) \, d\mu = - \int_{\R ^d } f\, \LL g \, d\mu = - \int_{\R ^d } \LL f \, g \, d\mu = \int_{\R ^d } (\nabla f) ^T \, \nabla g \, d\mu ,
$$
is non-positive on $\C _0 ^\infty (\R ^d , \R)$. By completeness, the operator is essentially self-adjoint, i.e. it admits a unique self-adjoint extension (still denoted $\LL$) with domain $\D (\LL) \subset L^2 (\mu)$ in which the space $\C _0 ^\infty (\R ^d , \R)$ is dense for the norm induced by $\LL$,
$$
\Vert f\Vert _{\D (L)} := \sqrt{\Vert f \Vert ^2 _{L^2 (\mu)} + \Vert L f \Vert ^2 _{L^2 (\mu)}}.
$$
The closure $(\cE _\mu , \D (\cE _\mu))$ of the bilinear form $(\cE _\mu,\C _0 ^\infty (\R ^d , \R))$ is a Dirichlet form on $L^2(\mu)$ and by the spectral theorem we have the dense inclusion $\D (\LL) \subset \D (\cE _\mu)$ for the norm induced by $\cE_\mu$,
$$
\Vert f\Vert _{\D (\cE _\mu)} := \sqrt{\Vert f \Vert ^2 _{L^2 (\mu)} + \Vert \vert \nabla f \vert \Vert ^2 _{L^2 (\mu)}}.
$$
In particular the form domain $\D (\cE _\mu)$ coincides with the Sobolev space $H^1 (\mu)$. \smallskip

In terms of the spectrum $\sigma (-L)$ of the operator $-L$, recall that it is divided into two parts: the essential spectrum, that is, the set of limit points in $\sigma (-L)$ and eigenvalues with infinite multiplicity, and the discrete spectrum, i.e., the complement of the essential spectrum consisting of isolated eigenvalues with finite multiplicity. All the elements of the spectrum, called in the sequel eigenvalues by abuse of language, are counted with their multiplicity. Then the Courant-Fischer min-max theorem \cite{RS} gives us the variational formulae for the eigenvalues of the operator $-L$. More precisely, the first eigenvalue is $\lambda_0 = 0$, possibly embedded in the essential spectrum, the constants being the associated eigenfunctions. Denoting $\perp$ the scalar product in $L^2(\mu)$, if we set for every $n \in \N^*$,
\begin{equation}
\label{eq:courant}
\lambda_{n} = \sup_{g_0 , g_1 , \ldots, g_{n-1} \in L^2 (\mu)} \, \underset{{\underset{f \perp g_i, \, i = 0, \ldots, n-1}{f\in \D(L)}}}{\inf} \frac{- \int_{\R ^d} f\, Lf \, d\mu}{\int_{\R ^d} f^2 \, d\mu} ,
\end{equation}
then either $\lambda_n$ is located below the bottom of the essential spectrum and thus it is actually the $n^{th}$ positive eigenvalue of the discrete spectrum or it is itself the bottom of the essential spectrum, all the $\lambda_m$ coinciding with $\lambda_n$ when $m \geq n$, and there are at most $n-1$ positive eigenvalues in the discrete spectrum below. The supremum is realized when the $g_i$ are the associated eigenfunctions and this is the case at least if the spectrum is discrete, for instance when the potential $V$ is uniformly convex. Note moreover that as usual the infimum above might be taken over $\C_0 ^\infty (\R^d , \R)$ instead of $\D(L)$ since $\C_0 ^\infty (\R^d , \R)$ is dense in $\D(L)$ for the norm induced by $L$ and thus for that induced by $\cE_\mu$. \smallskip

The variational formula \eqref{eq:courant} leads to the following observation: if the following Poincar\'e-type inequality with constant $\lambda >0$ holds,
\begin{equation}
\label{eq:poincare}
\lambda \, \int_{\R^d} f^2 \, d\mu \leq - \int_{\R^d} f\, L f \, d\mu ,
\end{equation}
for every function $f\in \C_0 ^\infty (\R^d , \R)$ such that $f\perp g_i$ for some functions $g_i$ in $L^2(\mu)$, $i=1,\ldots, n-1$, $g_0$ being constant, then we have the lower estimate
$$
\lambda_{n} \geq \lambda .
$$

As usual, the case $n=1$ corresponds to the classical Poincar\'e inequality and involves $\lambda_1 \, (= \lambda_1 - \lambda_0)$, the so-called spectral gap of the operator $-L$, governing the exponential speed of convergence in $L^2(\mu)$ of the semigroup $(e^{t L})_{t\geq 0}$: for every function $f\in \C_0 ^\infty (\R^d , \R)$ centered, i.e. such that $\mu (f) := \int_{\R^d} f \, d\mu =0$ (in other words, $f\perp 1$),
\begin{equation}
\label{eq:expo_cv}
\Vert e^{tL} f \Vert _{L^2 (\mu)} \leq e^{-\lambda_1 t} \, \Vert f \Vert _{L^2 (\mu)} .
\end{equation}
For instance there exists a spectral gap as soon as the potential V is convex, cf \cite{kls,bobkov}. This is also the case when $V$ is not convex but only convex at infinity, at the price of a perturbation argument, cf. \cite{ledoux}. In particular the spectral gap is related to the regularity of the solution to the Poisson equation as follows: for every $f\in \C_0 ^\infty (\R^d)$ centered, the Poisson equation
\begin{equation}
\label{eq:pois}
L g = f,
\end{equation}
admits a unique smooth centered solution $g  = - ( - L)^{-1} f \in \D (L)$. Here, the operator $(-L)^{-1}$ or more generally $(-L)^{-\alpha}$ for $\alpha>0$ is well-defined as a Riesz-type potential:
\begin{equation}
\label{eq:Riesz}
(-L)^{-\alpha} := \frac{1}{\Gamma(\alpha)} \, \int_0^{+\infty} t^{\alpha-1} \, e^{t L} \, dt,
\end{equation}
where $\Gamma$ is the famous Gamma function $\Gamma (\alpha) := \int_0^{+\infty} t^{\alpha-1} e^{-t} \, dt$, cf. e.g. \cite{BGL}, and the inequality \eqref{eq:expo_cv} leads to the boundedness in $L^2 (\mu)$ of the operator $(-L)^{-\alpha}$, $\alpha >0$. This argument will be at the heart of the forthcoming analysis with the $L^2$ approach of H\"ormander \cite{hormander}.

\subsection{Intertwinings}
Now let us turn our attention to the notion of intertwining studied in the recent paper \cite{ABJ}. Denote $\L$ the diagonal matrix operator
\[
\L = \left(
\begin{array}{ccc}
\LL & & \\
& \ddots & \\
& & \LL
\end{array}
\right) ,
\]
which acts naturally on the space $\C ^\infty (\R^d,\R^d)$ of smooth vector fields $F:\R^d \to \R^d$. Letting $A : \R^d \to \mathcal{M}_{d\times d} (\R)$ be a smooth invertible matrix, we denote the matrix
$$
M_A := A \, \nabla ^2 V \, A^{-1} - A \, \L A^{-1}.
$$
We assume furthermore that the following equivalent assertions are satisfied: if we denote
the positive-definite $S = (A A^T)^{-1}$, then
\[
\left \{
\begin{array}{l}
\mbox{ the matrix } (A^{-1}) ^T \, \nabla A^{-1} \mbox{ is symmetric}; \\

\mbox{ the matrix } (A^{-1}) ^T \, \L A^{-1} \mbox{ is symmetric}; \\

\mbox{ the matrix } S \, M_A \mbox{ is symmetric}; \\

\mbox{ the matrix } A^{-1} \, M_A \, A \mbox{ is symmetric}.
\end{array}
\right.
\]

\noindent Then if $L^2 (S,\mu)$ denotes the space consisting of vector fields $F$ such that
$$
\Vert F \Vert _{L^2(S,\mu)} := \sqrt{\int_{\R^d} F^T \, S \, F \, d\mu } < \infty ,
$$
then the operator $\L_A$ defined as
$$
\L_A F = \L F + 2 A \, \nabla A^{-1} \, \nabla F ,
$$
is symmetric and non-positive on $\C_0 ^\infty (\R^d,\R^d) \subset L^2 (S,\mu)$, the subspace of $\C ^\infty (\R^d,\R^d)$ of compactly supported vector fields: for every $F,G \in \C_0 ^\infty (\R^d,\R^d)$,
$$
\int_{\R^d} (\L_A F)^T \, S \, G \, d\mu = \int_{\R^d} F^T \, S \, \L_A G \, d\mu = - \int_{\R^d} (\nabla F)^T \, S \, \nabla G \, d\mu.
$$
Above, the gradients act by contraction as follows: if $A^{-1} = (a^{i,j})_{i,j = 1, \ldots, d}$ then $\nabla A^{-1} \, \nabla F$ is a vector field defined by
$$
\left( \nabla A^{-1} \, \nabla F \right) _i := \sum_{j,k = 1} ^d \partial _k a^{i,j} \, \partial_k F_j , \quad i \in \{ 1,\ldots , d \},
$$
and moreover,
$$
(\nabla F) ^T \, S \, \nabla G := \sum_{i,j,k = 1} ^d \partial_k F_i \, S_{i,j} \, \partial_k F_j .
$$
Then the operator $\L_A$ is essentially self-adjoint on $\C_0 ^\infty (\R^d,\R^d)$ and admits a unique extension (still denoted $\L_A$) with domain $\D(\L_A ) \subset L^2 (S,\mu)$. Such a property also holds for the Schr\"odinger-type operator $\L_A ^{M_A} := \L_A - M_A$, under the additional assumption that the matrix
$$
A^{-1} \, M_A \, A = \nabla ^2 V - \L A^{-1} \, A ,
$$
is uniformly bounded from below. We still denote $\L_A ^{M_A}$ its extension with domain $\D(\L_A ^{M_A}) \subset L^2 (S,\mu)$. \smallskip

Once the various protagonists of interest have been introduced, we recall the intertwining relation emphasized in \cite{ABJ}: for every $f\in \C^\infty (\R^d,\R)$,
\begin{equation}
\label{eq:entrelac_A}
A \, \nabla L f = \L_A ^{M_A} (A \, \nabla f).
\end{equation}
If we define the subspace $\nabla_A = \{ A \nabla f : f\in \C_0 ^\infty (\R^d) \} \subset \C_0 ^\infty (\R^d,\R^d)$ of weighted gradients, then the previous analysis can be carried on for the restrictions of the operators $\L_A$ and $\L_A ^{M_A}$ to $\nabla_A$, the corresponding self-adjoint extensions being denoted $\L_A \vert _{\nabla_A}$ and $\L_A ^{M_A} \vert _{\nabla_A}$ and having respective domains
$$
\D (\L_A \vert _{\nabla_A}) = \D (\L_A ) \cap \{ A \nabla f : f\in \D (L) \},
$$
and
$$
\D (\L_A ^{M_A} \vert _{\nabla_A}) = \D (\L_A ^{M_A}) \cap \{ A \nabla f : f\in \D (L) \}.
$$
Note however that the restricted operator $\L_A ^{M_A} \vert _{\nabla_A}$ is symmetric and non-positive even if the four equivalent assertions above are not assumed, a property directly inherited from that of the operator $L$, according to the intertwining \eqref{eq:entrelac_A}. \smallskip

Before turning the Brascamp-Lieb inequalities, we mention that a convenient choice of matrix $A$, which will be used in the sequel, is given by
$$
(A^T)^{-1} = J_H ,
$$
where $H \in \C^\infty (\R^d,\R^d)$ is a diffeomorphism and $J_H$ denotes the associated Jacobian matrix. Indeed we have in this case
$$
A^{-1} \, M_A \, A = \nabla^2 V - \L J_H ^T \, (J_H ^T)^{-1},
$$
and some computations allows us to rewrites it as the matrix
$$
 - J_{\L H} ^T \, (J_H ^T )^{-1},
$$
which is the multi-dimensional version of the practical criterion emphasized in the one-dimensional case by the authors in \cite{bj,bjm2}. Such an identity will be used many times in the sequel. Note however that this matrix has no reason \textit{a priori} to be symmetric and therefore one has to choose carefully the diffeomorphism $H$ to ensure this property.

\section{Brascamp-Lieb inequalities and spectral estimates}
\label{sect:BL}
As observed in \cite{ABJ}, the intertwining approach has many interesting consequences in terms of functional inequalities and among them Brascamp-Lieb type inequalities of first and second orders, to which this section is devoted to.

\subsection{First-order Brascamp-Lieb inequalities}
Recall first that the classical Brascamp-Lieb inequality \cite{brascamp_lieb} stands as follows: if the matrix $\nabla^2 V$ is positive-definite, then for every sufficiently smooth function $f$, we have
\begin{equation}
\label{eq:BL}
\Var _\mu (f ) \leq \int_{\R^d} (\nabla f)^T \, (\nabla^2 V) ^{-1} \, \nabla f \, d\mu ,
\end{equation}
where $\Var_\mu (f)$ denotes the variance of $f$ under $\mu$, i.e.
$$
\Var_\mu (f) := \mu (f^2) - \mu (f) ^2 .
$$
In contrast to the classical Poincar\'e inequality, such an inequality always admits extremal functions given by $f = c^T \, \nabla V + m$ with $c\in \R^d$ and $m\in \R$ some constants. \smallskip

Before stating the main results of the present paper, let us first introduce a key lemma. The idea is to write conveniently the variance by using the intertwining \eqref{eq:entrelac_A}, in the spirit of the $L^2$ method of H\"ormander \cite{hormander}. Since the assumptions in force below do not ensure automatically a spectral gap, we have to assume the existence of a solution to the Poisson equation \eqref{eq:pois}.
\begin{lemme}
\label{lemme:var}
Assume that the matrix $\nabla^2 V - \L A^{-1} \, A $ is symmetric and uniformly bounded from below. Assume moreover that for every centered $f \in \C _0 ^\infty (\R^d)$, there exists a unique smooth centered solution $g\in \D (L)$ to the Poisson equation $f = L g$. Then we have the identity:
\begin{eqnarray*}
\Var_\mu (f) & = & - 2 \, \int_{\R^d} (A\nabla f)^T \, S \, A \nabla g \, d\mu + \int_{\R^d} (A\nabla g )^T \, S \, \L_A ^{M_A} \, (A\nabla g) \, d\mu .
\end{eqnarray*}
\end{lemme}
\begin{proof}
We have
\begin{eqnarray*}
\Var_\mu (f) & = & 2 \, \Var_\mu (f) - \Var_\mu (f) \\
& = & 2 \, \int_{\R^d} f \, L g \, d\mu - \int_{\R^d} (L g)^2 \, d\mu \\
& = & - 2 \, \int_{\R^d} (\nabla f)^T \, \nabla g \, d\mu + \int_{\R^d} (\nabla g)^T \, \nabla L g \, d\mu \\
& = & - 2 \, \int_{\R^d} (A\nabla f)^T \, S \, A\nabla g \, d\mu + \int_{\R^d} (A \nabla g)^T \, S \, A\nabla L g \, d\mu \\
& = & - 2 \, \int_{\R^d} (A\nabla f)^T \, S \, A\nabla g \, d\mu + \int_{\R^d} (A \nabla g)^T \, S \, \L_A ^{M_A} \, (A\nabla g) \, d\mu ,
\end{eqnarray*}
where in the last line we used the intertwining \eqref{eq:entrelac_A}.
\end{proof}

Actually, this key lemma enables to recover briefly the Brascamp-Lieb type inequality obtained in \cite{ABJ}, which can be seen as an extension of the classical Brascamp-Lieb inequality \eqref{eq:BL} corresponding to the choice below of the matrix $A=I$. Beyond its simplicity, let us give a short proof of this inequality to see how the use of Lemma \ref{lemme:var} brings a new point of view. We assume that the matrix $\nabla^2 V - \L A^{-1} \, A $ is symmetric and positive-definite. At the price of an approximation procedure somewhat similar to that emphasized in \cite{ABJ}, we might assume the existence of a spectral gap for the operator $-L$ so that letting $f\in \C^\infty _0 (\R^d,\R)$, we can use the following identity provided by the intertwining \eqref{eq:entrelac_A} and involving the solution $g$ to the Poisson equation \eqref{eq:pois}:
$$
A \nabla f + M_A \, A\nabla g = \L_A (A\nabla g).
$$
Using then Lemma \ref{lemme:var} and noticing that
$$
A^T \, S \, M_A^{-1} \, A = \left( \nabla^2 V - \L A^{-1} \, A \right)  ^{-1} ,
$$
we have
\begin{eqnarray*}
\Var_\mu (f) & = & - 2 \, \int_{\R^d} (A\nabla f)^T \, S \, A\nabla g \, d\mu + \int_{\R^d} (A \nabla g)^T \, S \, \L_A ^{M_A} \, (A\nabla g) \, d\mu \\
& = & \int_{\R^d} (A\nabla g)^T \, S \, \L_A (A\nabla g) \, d\mu + \int_{\R^d} (A \nabla f)^T \, S \, M_A ^{-1} \, A\nabla f \, d\mu \\
& & - \int_{\R^d} (A \nabla f + M_A \, A\nabla g)^T \, S \, M_A ^{-1} \, (A \nabla f + M_A \, A\nabla g) \, d\mu  \\
& = & \int_{\R^d} (A\nabla g)^T \, S \, \L_A (A\nabla g) \, d\mu + \int_{\R^d} (\nabla f)^T \, \left( \nabla^2 V - \L A^{-1} \, A \right)  ^{-1} \, \nabla f \, d\mu \\
 & & - \int_{\R^d} \left(A^{-1} \, \L_A (A\nabla g)\right)^T \,  \left( \nabla^2 V - \L A^{-1} \, A \right)  ^{-1}  \, A^{-1} \, \L_A (A\nabla g) \, d\mu .
\end{eqnarray*}
Finally, since the matrix $\left( \nabla^2 V - \L A^{-1} \, A \right)  ^{-1}$ is positive definite and the operator $ \L_A$ is non-positive on $L^2(S,\mu)$, the first and third terms are non-positive and we then obtain the following Brascamp-Lieb type inequality: for every $f\in \C_0 ^\infty (\R^d , \R)$,
\begin{equation}
\label{eq:BL_A}
\Var _\mu (f ) \leq \int_{\R^d} (\nabla f)^T \, \left( \nabla^2 V - \L A^{-1} \, A \right) ^{-1} \, \nabla f \, d\mu ,
\end{equation}
which is the inequality derived in \cite{ABJ}. In particular if the smallest eigenvalue of this matrix is uniformly bounded from below by some positive constant, then we obtain the following estimate on the spectral gap:
\begin{equation}
\label{eq:gap}
\lambda_1 \geq \inf\, \rho \left( \nabla^2 V - \L A^{-1} \, A \right).
\end{equation}
Above, for a given symmetric positive-definite matrix $M$, the quantity $\rho (M)$ denotes its smallest eigenvalue. \smallskip

Coming back to our initial remark, we observe that the non-positive remainder terms which were ignored to get the Brascamp-Lieb type inequality \eqref{eq:BL_A} enables us to identify more carefully the potential extremal functions, an issue left open in \cite{ABJ}. Indeed the equality in \eqref{eq:BL_A} holds if and only if the vector field $A \nabla g$ is an eigenfunction associated to the eigenvalue 0 for the operator $-\L _A$, meaning that $A \nabla g$ is a constant vector, say $c\in \R^d$, and thus it yields by the intertwining \eqref{eq:entrelac_A},
$$
\nabla f = A^{-1} \, \L_A ^{M_A} \, (A\nabla g) = - \left( \nabla^2 V - \L A^{-1} \, A \right) \, \nabla g,
$$
which has no reason a priori to be a gradient. Hence it may happen that there is no extremal function $f$ saturating \eqref{eq:BL_A} without assuming any other additional condition on the matrix $A$. As mentioned at the end of Section \ref{sect:prelim}, a convenient choice to go one step beyond is to set $A = (J_H ^T)^{-1}$ where $H \in \C ^\infty (\R^d,\R^d)$ is some diffeomorphism such that the matrix
$$
\nabla^2 V - \L J_H ^T \, (J_H ^T)^{-1} =  - J_{\L H} ^T \, (J_H ^T )^{-1},
$$
is symmetric and positive-definite. Then we obtain
$$
g = c^T \, H - \int_{\R^d} c^T \, H \, d\mu \quad \mbox{ and thus } \quad f = c^T \, \L H + \mu (f).
$$
In other words, under some technical conditions on the diffeomorphism $H$, the extremal functions of the Brascamp-Lieb type inequality \eqref{eq:BL_A} are of the type $c^T \, \L H + m$ with $c\in \R^d$ and $m\in \R$, the choice $H(x) = x$ recovering the classical Brascamp-Lieb inequality \eqref{eq:BL}. \smallskip

\subsection{Second-order Brascamp-Lieb inequalities and spectral estimates}
As mentioned in the Introduction, there have been many extensions of the Brascamp-Lieb inequality \eqref{eq:BL} in the recent years, which can take different forms and can be established by various methods, our previous inequality \eqref{eq:BL_A} being an example of such inequalities. Among all the interesting references on the topic, the recent article of Cordero-Erausquin \cite{cordero} exhibits an improvement of the Brascamp-Lieb inequality \eqref{eq:BL} under an additional orthogonality assumption. More precisely, if the matrix $\nabla^2 V$ is positive-definite, then for every $f \in \C_0 ^\infty (\R^d , \R)$ such that
$$
\Cov_\mu (f,id) = 0,
$$
where $id$ is the identity vector field $id (x) = x$ on $\R^d$ and $\Cov_\mu (f,id)$ is the vector of $\R^d$ whose coordinates are the covariances between the function $f$ and the coordinates $x_i$, then we have the second-order Brascamp-Lieb inequality
\begin{equation}
\label{eq:cordero}
\Var_\mu (f) \leq \int_{\R^d} (\nabla f)^T \, \left( \nabla^2 V + \lambda_1 \, I \right) ^{-1} \, \nabla f \, d\mu .
\end{equation}
As a direct consequence, the Courant-Fisher theorem entails the following spectral estimate
\begin{equation}
\label{eq:cordero_eigen}
\lambda_{d+1} \geq \lambda_1 + \inf \, \rho (\nabla^2 V ),
\end{equation}
permitting when $V$ is uniformly convex (recall that the spectrum is discrete in this case) a control from below of a higher order eigenvalue, namely the $(d+1)^{th}$ positive eigenvalue $\lambda_{d+1}$. In particular, the fact that we obtain an estimate beyond the spectral gap, corresponding to a first order, justifies the terminology ``second-order" to qualify this Brascamp-Lieb type inequality. The eigenvalue estimate \eqref{eq:cordero_eigen} is optimal for the standard Gaussian distribution, i.e. for the potential $V = \vert \cdot \vert ^2 /2$, since $\lambda_1 = 1$ is of multiplicity $d$ and $\lambda_{d+1} = 2$. The presence of the eigenvalue $\lambda_{d+1}$ is relevant since in this convex setting it is different from $\lambda_1$, the multiplicity of the latter being at most the dimension $d$, as noticed by Barthe and Cordero-Erausquin in \cite{barthe-cordero}, following an argument of Klartag \cite{klartag}. Note however that the multiplicity of $\lambda_1$ has no reason to be systematically equal to $d$, as it might be observed for a centered Gaussian distribution with independent coordinates of different variances. We mention that we will come back to this problem of maximal multiplicity in a moment. \smallskip

In the sequel, we generalize via the intertwining approach Cordero-Erausquin's inequality \eqref{eq:cordero} by obtaining more general second-order Brascamp-Lieb inequalities, together with a lower bound on $\lambda_{d+1}$ allowing us to consider examples beyond the case of a uniformly convex potential $V$ (in particular the essential spectrum might not be empty). It corresponds to the two main contributions of the present paper. \smallskip

Actually, the proof of the inequality \eqref{eq:BL_A} leads us to consider more carefully the non-positive term
$$
\int_{\R^d} (A\nabla g)^T \, S \, \L_A (A\nabla g) \, d\mu .
$$
Hence the Brascamp-Lieb type inequality \eqref{eq:BL_A} will be improved as soon as we are able to bound from above this term by some non-positive quantity, and this is the matter of the forthcoming analysis. The price to pay for such an improvement is to assume the existence of a spectral gap for the operator $-\L_A$, or more precisely for its restriction $-\L_A \vert _{\nabla_A}$ to the subspace of weighted gradients to which we turn now. \smallskip

First let us consider the notion of mean in the space $L^2 (S,\mu)$. We assume in the sequel that the matrix $\int_{\R^d} S \, d\mu$ defined by
 \[
\left( \int_{\R^d} S \, d\mu\right)_{i,j} = \int_{\R^d} S_{i,j} \, d\mu, \quad 1 \leq i,j \leq d ,
\]
is well-defined and invertible. Now, for every $F \in L^2 (S,\mu)$, the mean $m_S (F)$ of the vector field $F$ is defined as $$
m_S (F) = \left( \int_{\R^d} S \, d\mu\right) ^{-1} \, \int_{\R^d} S \, F \, d\mu ,
$$
and is the unique constant vector of $\R^d$ satisfying
$$
m_S (F) = \arg \min_{c\in \R^d}  \int_{\R^d} (F-c)^T \, S \, (F-c) \, d\mu .
$$
In particular, every constant vector $c \in \R^d$ belongs to $L^2(S,\mu)$ and we have $m_S (c) = c$. Moreover, for every $F,G \in L^2 (S,\mu)$,
$$
\int_{\R^d} (F-m_S (F))^T \, S \, (G-m_S (G)) \, d\mu = \int_{\R^d} F^T \, S \, G \, d\mu - m_S (F) ^T \, \left( \int_{\R^d} S \, d\mu \right) \, m_S (G) .
$$
In the sequel, for a given vector field $F\in L^2 (S,\mu)$, we set
$$
\widetilde{F} = F - m_S (F) ,
$$
which is thus centered in $L^2 (S,\mu)$, i.e. for every $c\in \R^d$,
$$
\int_{\R^d} c^T \, S \, \widetilde{F} \, d\mu = \int_{\R^d} \widetilde{F} ^T \, S \, c \, d\mu = 0.
$$
We are now in position to define the spectral gap in $L^2 (S,\mu)$ of the operator $-\L_A$ by the variational formula
\begin{eqnarray*}
\lambda_1 ^A & = & \inf \left \{ \frac{- \int_{\R^d} \widetilde{F}^T \, S \, \L_A \widetilde{F} \, d\mu}{\int_{\R^d} \widetilde{F} ^T \, S \, \widetilde{F} \, d\mu} :  F \in \D (\L_A) \right \} \\
& = & \inf \left \{ \frac{- \int_{\R^d} F^T \, S \, \L_A F \, d\mu}{\int_{\R^d} F^T \, S \, F \, d\mu} :  F \in \D (\L_A), \, \int_{\R^d} S \, F \, d\mu = 0 \right \}.
\end{eqnarray*}

Actually, except for practical issues as we will see next, the quantity of interest to consider is not the spectral gap of $-\L_A$ but the spectral gap of its restriction $-\L_A \vert _{\nabla_A}$ acting on the space of weighted gradients, i.e.
\begin{eqnarray*}
\lambda_1 ^A \vert _{\nabla_A} & := & \inf \left \{ \frac{- \int_{\R^d} \widetilde{F}^T \, S \, \L_A \widetilde{F} \, d\mu}{\int_{\R^d} \widetilde{F} ^T \, S \, \widetilde{F} \, d\mu} :  F \in \D (\L_A \vert _{\nabla_A}) \right \} \\
& = & \inf \left \{ \frac{- \int_{\R^d} F^T \, S \, \L_A F \, d\mu}{\int_{\R^d} F^T \, S \, F \, d\mu} :  F \in \D (\L_A \vert _{\nabla_A}), \, \int_{\R^d} S \, F \, d\mu = 0 \right \}.
\end{eqnarray*}
In particular, we obtain a lower bound on $\lambda_1 ^A \vert _{\nabla_A}$, say $\lambda>0$, as soon as the following Poincar\'e inequality holds: for every $f \in \C_0 ^\infty (\R^d,\R)$,
\begin{equation}
\label{eq:poincare_A}
\lambda \, \int_{\R^d} \widetilde{A \nabla f } ^T \, S \, \widetilde{A \nabla f } \, d\mu \leq - \int_{\R^d} \widetilde{A\nabla f }^T \, S \, \L_A (\widetilde{A \nabla f }) \, d\mu .
\end{equation}
Note that the two spectral gaps $\lambda_1 ^A$ and $\lambda_1 ^A \vert _{\nabla_A}$ have no reason to coincide, $\lambda_1 ^A \vert _{\nabla_A}$ being only larger than $\lambda_1 ^A$ by its very definition. \smallskip

We are now able to state the first main result of the paper, corresponding to a second-order Brascamp-Lieb inequality. Once again, since the matrix $\nabla^2 V - \L A^{-1} \, A$ is only assumed below to be symmetric and positive-definite, the spectral gap $\lambda_1$ might be zero and thus we have to assume a priori the existence of a solution to the Poisson equation \eqref{eq:pois}.
\begin{theo}
\label{theo:main2}
Assume that the matrix $\nabla^2 V - \L A^{-1} \, A$ is symmetric and positive-definite. Moreover we assume that $\lambda_1 ^A \vert _{\nabla _A} >0$. Let $f\in \C_0 ^\infty (\R^d)$ be centered and assume that there exists a unique smooth centered solution $g \in \D (L)$ to the Poisson equation $f=Lg$. Then the following second-order Brascamp-Lieb inequality holds:
\begin{eqnarray*}
\Var_\mu (f) & \leq & \int_{\R^d} (\nabla f )^T \, \left( \lambda_1 ^A \vert _{\nabla _A} \, I + \nabla^2 V - \L A^{-1} \, A \right) ^{-1} \, \nabla f \, d\mu \\
& & + \int_{\R^d} m_S (A \nabla g )^T \, S \, M_A \, m_S (A\nabla g) \, d\mu  .
\end{eqnarray*}
\end{theo}
\begin{proof}
%The second inequality being a direct consequence, let us only prove the first one.
By Lemma \ref{lemme:var} and the Poincar\'e inequality \eqref{eq:poincare_A}, we have
\begin{eqnarray*}
\Var _\mu (f) & = & - 2 \, \int_{\R^d} (A\nabla f)^T \, S \, \widetilde{A\nabla g} \, d\mu + \int_{\R^d} (A\nabla g ) ^T \, S \, \L_A ^{M_A} (A\nabla g) \, d\mu \\
& & - 2 \, \int_{\R^d} (A\nabla f)^T \, S \, m_S(A\nabla g) \, d\mu \\
& \leq & - 2 \, \int_{\R^d} (A\nabla f )^T \, S \, \widetilde{A\nabla g} \, d\mu - \int_{\R^d} \widetilde{A \nabla g} ^T \, S \, \left( \lambda_1 ^A \vert _{\nabla _A} \, I + M_A \right) \widetilde{A \nabla g} \, d\mu \\
& & - \int_{\R^d} m_S (A \nabla g) ^T S \, M_A \, m_S (A \nabla g) \, d\mu - 2 \, \int_{\R^d} m_S (A \nabla g) ^T S \, M_A \, \widetilde{A \nabla g} \, d\mu \\
 & & - 2 \, \int_{\R^d} (A\nabla f)^T \, S \, m_S(A\nabla g) \, d\mu \\
& = & \int_{\R^d} (A\nabla f )^T \, S \, \left( \lambda_1 ^A \vert _{\nabla _A} \, I + M_A \right)^{-1} \, (A\nabla f ) \, d\mu \\
& & - \int_{\R^d} \Theta ^T \, S \, \left( \lambda_1 ^A \vert _{\nabla _A} \, I + M_A \right) ^{-1} \, \Theta \, d\mu - \int_{\R^d} m_S (A \nabla g) ^T S \, M_A \, m_S (A \nabla g) \, d\mu \\
& & - 2 \, \int_{\R^d} m_S (A \nabla g) ^T S \, M_A \, \widetilde{A \nabla g} \, d\mu - 2 \, \int_{\R^d} (A\nabla f)^T \, S \, m_S(A\nabla g) \, d\mu ,
\end{eqnarray*}
where $\Theta$ is defined by the quantity
\begin{eqnarray*}
\Theta & := & A\nabla f + (\lambda_1 ^A \vert _{\nabla _A} \, I + M_A) \, \widetilde{A\nabla g} \\
& = & (\L_A + \lambda_1 ^A \vert _{\nabla _A} \, I ) \, (\widetilde{A\nabla g}) - M_A \, m_S (A\nabla g),
\end{eqnarray*}
according to the intertwining \eqref{eq:entrelac_A}. Now, we have
\[
 A^T \, S \, \left( \lambda_1 ^A \vert _{\nabla _A} \, I + M_A \right) ^{-1} \, A = \left( \lambda_1 ^A \vert _{\nabla _A} \, I + \nabla^2 V - \L A^{-1} \, A \right) ^{-1} ,
\]
and thus
\begin{eqnarray*}
\int_{\R^d} \Theta ^T \, S \, \left( \lambda_1 ^A \vert _{\nabla _A} \, I + M_A \right) ^{-1} \, \Theta \, d\mu & = & \int_{\R^d} (A^{-1} \, \Theta) ^T \, \left( \lambda_1 ^A \vert _{\nabla _A} \, I + \nabla^2 V \right. \\
& & \left. - \L A^{-1} \, A \right) ^{-1} \, A^{-1} \, \Theta \, d\mu \\
& \geq & 0.
\end{eqnarray*}
Finally, we note that using once again the intertwining \eqref{eq:entrelac_A},
\begin{eqnarray*}
\int_{\R^d} (A\nabla f)^T \, S \, m_S(A\nabla g) \, d\mu & = & \int_{\R^d} m_S (A\nabla g)^T \, S \, A\nabla f \, d\mu \\
& = & \int_{\R^d} m_S (A\nabla g)^T \, S \, \L _A ^{M_A} \, (A\nabla g) \, d\mu \\
& = & - \int_{\R^d} m_S (A\nabla g)^T \, S \, M_A \, \widetilde{A\nabla g} \, d\mu \\
& & - \int_{\R^d} m_S (A\nabla g)^T \, S \, M_A \, m_S (A\nabla g) \, d\mu ,
\end{eqnarray*}
which achieves the proof.
\end{proof}
In the classical case $A = I$ or, in other words, $H = id$, Theorem \ref{theo:main2} and its proof yield the following inequality: for every $f\in \C _0 ^\infty (\R^d,\R)$,
\begin{eqnarray}
\label{eq:cordero_remainder}
\nonumber \Var_\mu (f) & \leq & \int_{\R^d} (\nabla f)^T \, \left( \nabla^2 V + \lambda_1 I \right) ^{-1} \, \nabla f \, d\mu + \int_{\R^d} c^T \, \nabla^2 V \, c \, d\mu \\
& & - \int_{\R^d} \Theta ^T \, \left( \nabla^2 V + \lambda_1 I \right) ^{-1} \, \Theta \, d\mu  ,
\end{eqnarray}
where we recall that $\Theta$ is defined by
$$
\Theta := (\L +\lambda_1 \, I ) \, (\nabla g - c) - \nabla^2 V \, c , \quad \mbox{ with } \quad c := \int_{\R^d} \nabla g \, d\mu .
$$
Since the constant $c$ can be rewritten in terms of the function $f$ as $c = - \Cov_\mu (f,id)$, we recover and slightly improve with a remainder term Cordero-Erausquin's inequality \eqref{eq:cordero} in the case of the centering $c=0$. Actually, it seems difficult to remove the centering in his inequality (the reason is that the induced scalar product $( \nabla^2 V + \lambda_1 \, I) ^{-1}$ in the right-hand-side of \eqref{eq:cordero} depends on the space variable)  whereas our inequality \eqref{eq:cordero_remainder} does not require any centering. Note that in the uniformly convex case, the different centering
\begin{equation}
\label{eq:centering}
\int_{\R^d} \nabla f \, d\mu = 0,
\end{equation}
appeared in a first paper of Cordero-Erausquin and his coauthors \cite{CFM} and then in Harg\'e's work \cite{harge}, both to establish some second-order Poincar\'e inequalities. Actually, the gradients $\nabla f$ and $\nabla g$ have no reason a priori to be centered simultaneously, except in the standard Gaussian case since we always have
$$
\int_{\R^d} \nabla f \, d\mu = - \int_{\R^d} \nabla^2 V \, \nabla g \, d\mu .
$$
Moreover, it is also the case when the measure $\mu$ is unconditional (i.e. its Lebesgue density is symmetric with respect to any coordinate hyperplane) as well as the function $f$, as noticed by Barthe and Cordero-Erausquin in \cite{barthe-cordero}. \smallskip

In the standard Gaussian case, the inequality \eqref{eq:cordero_remainder} rewrites as
$$
\Var_\mu (f) \leq \frac{1}{2} \, \int_{\R^d} \vert \nabla f  \vert ^2 \, d\mu + \vert c\vert ^2 -  \frac{1}{2} \, \int_{\R^d} \vert (\L + I ) \, (\nabla g - c) - c  \vert ^2 \, d\mu ,
$$
and after expanding the square in the last integral, we get the inequality
\begin{equation}
\label{eq:gauss}
\Var_\mu (f) \leq \frac{1}{2} \, \int_{\R^d} \vert \nabla f  \vert ^2 \, d\mu + \frac{1}{2} \, \left \vert \int_{\R^d} \nabla f \, d\mu \right \vert ^2 - \frac{1}{2} \, \int_{\R^d} \vert (\L + I ) \, (\nabla g - c) \vert ^2 \, d\mu ,
\end{equation}
which slightly improves the inequality of Goldstein-Nourdin-Peccati \cite{gnp} obtained directly by a simple spectral decomposition using Hermite polynomials. Note that the inequality \eqref{eq:gauss} might also be obtained by spectral decomposition, with equality if and only if $f$ is a Hermite polynomial of degree one, two or three. In particular, it would be interesting to compare \eqref{eq:gauss} with the dimensional dependent inequalities appearing in the literature, namely with the two following inequalities obtained through the Borrell-Brascamp-Lieb approach:

$\circ$ Bobkov-Ledoux's inequality \cite{bob_ledoux2} :
$$
\Var_\mu (f) \leq 6 \, \int_{\R^d} \vert \nabla f  \vert ^2 \, d\mu - 6 \, \int_{\R^d} \frac{ \vert (\nabla f) ^T \, x \vert ^2}{d + \vert x\vert ^2} \, d\mu ,
$$

$\circ$ Bolley-Gentil-Guillin's inequality \cite{bgg} :
$$
\Var_\mu (f) \leq \int_{\R^d} \vert \nabla f  \vert ^2 \, d\mu - \int_{\R^d} \frac{\vert f - (\nabla f) ^T \, x \vert ^2}{d + \vert x\vert ^2} \, d\mu ,
$$

$\circ$ and also to an inequality we established recently in \cite{bjm} by exploiting the spherical invariance of the standard Gaussian distribution,
$$
\Var_\mu (f) \leq \frac{d(d+3)}{d-1} \, \int_{\R^d} \frac{\vert \nabla f  \vert ^2}{1+\vert x\vert ^2} \, d\mu .
$$
\smallskip

Coming back to the general situation of Theorem \ref{theo:main2}, we have under the centering condition $m_S (A\nabla g) = 0$ the tight inequality
\begin{eqnarray*}
\Var_\mu (f) & \leq & \int_{\R^d} (\nabla f )^T \, \left( \lambda_1 ^A \vert _{\nabla _A} \, I + \nabla^2 V - \L A^{-1} \, A \right) ^{-1} \, \nabla f \, d\mu .
\end{eqnarray*}
As we may observe from the proof, the optimality in Theorem \ref{theo:main2} is obtained if and only if the Poincar\'e inequality \eqref{eq:poincare_A} is saturated and if $\Theta=0$. The equality in \eqref{eq:poincare_A} holds when $\widetilde{A \nabla g}$ is an eigenfunction associated to the eigenvalue $\lambda_1 ^A \vert _{\nabla _A}$ for the restricted operator $- \L_A \vert _{\nabla_A}$. However, if it exists, we ignore its potential expression even for the choice $A = (J_H ^T)^{-1} $ for some convenient diffeomorphism $H \in \C^\infty (\R^d,\R^d)$, which revealed to be relevant for the first-order Brascamp-Lieb inequality \eqref{eq:BL_A} as we have seen previously. \smallskip

Note that the centering condition $m_S (A\nabla g) = 0$ essentially focuses on the function $g$, so that we cannot obtain directly an estimate on $\lambda_{d+1}$ since the orthogonality conditions should be required only on the function $f$. However once again the choice of the matrix $A = (J_H ^T)^{-1} $, where $H \in \C^\infty (\R^d,\R^d)$ is a diffeomorphism, allows us to solve this problem (we sometimes keep the notation $A$ below in order to avoid a heavy notation). We are now in position to state the second main result of the present paper. Recall that in this case we have
$$
\nabla^2 V - \L A^{-1} \, A = - J_{\L H} ^T \, (J_H ^T )^{-1}.
$$
\begin{theo}
\label{theo:multidim}
Let $H \in \C^\infty (\R^d,\R^d)$ be a diffeomorphism. Set $A = (J_H ^T)^{-1} $ and assume that the matrix $- J_{\L H} ^T \, (J_H ^T )^{-1}$ is symmetric and uniformly bounded from below by some positive constant. Moreover, assume that we have $\lambda_1 ^A \vert _{\nabla _A} >0$. Then we get the following spectral estimate:
\begin{eqnarray}
\label{eq:lambda_d}
\lambda_{d+1} & \geq & \lambda_1 ^A \vert _{\nabla _A} + \inf \, \rho ( - J_{\L H} ^T \, (J_H ^T )^{-1}) .
\end{eqnarray}
\end{theo}
\begin{proof}
The proof is straightforward: with this choice of matrix $A$, the centering condition $m_S (A\nabla g) = 0$ derived from Theorem \ref{theo:main2} can be rewritten as follows:
$$
0 = \int_{\R^d} (A^T)^{-1} \, \nabla g \, d\mu = \int_{\R^d} J_H \, \nabla g \, d\mu = - \int_{\R^d} H \, L g \, d\mu  = - \int_{\R^d} H \, f \, d\mu ,
$$
and using Theorem \ref{theo:main2}, the Courant-Fisher theorem entails the conclusion.
\end{proof}
Let us comment the potential optimality in the spectral estimate of Theorem \ref{theo:multidim}. To do so, let us consider first the case of a general invertible matrix $A$ such that the matrix $\nabla^2 V - \L A^{-1} \, A$ is symmetric and uniformly bounded from below by some positive constant. By construction, the two operators of Schr\"odinger-type $\L ^{\nabla^2 V}$ and $\L_A ^{M_A}$ are unitary equivalent, the multiplication by the matrix $A^{-1}$ being an unitary transformation from $L^2(S,\mu)$ to $L^2 (I ,\mu)$, and therefore their spectra coincide. The same argument applies for the restricted operators $\L ^{\nabla^2 V} \vert _{\nabla_I}$ and $\L_A ^{M_A} \vert _{\nabla_A}$. Such a property has already been noticed in the papers \cite{bj,ABJ}, dealing with the bottom of the spectra. From the probabilistic point of view, this transformation might be interpreted as a Doob's transformation: ``we multiply $A^{-1}$ inside the operator and by $A$ outside". However the relation between the operators $L$ and $ \L ^{\nabla^2 V}$ is more subtle. Indeed they are not directly unitary equivalent but if we restrict the first to the space orthogonal to constant functions, i.e. to $const  ^\perp := \D (\L) \cap \{ f \in L^2 (\mu) : f \perp 1 \}$ and we consider the restricted operator $\L ^{\nabla^2 V} \vert _{\nabla_I}$, then there exists a unitary transformation between them, a property enlighten by Johnsen in \cite{johnsen}. More precisely, if $U$ stands for the operator $U := \nabla (-L)^{-1/2}$ acting on $const  ^\perp$, with values in $\D (\L ^{\nabla^2 V} \vert _{\nabla_I})$ and which is well-defined as a Riesz-type potential \eqref{eq:Riesz}, then $U$ is a unitary mapping and we have the identity
$$
\L ^{\nabla^2 V} = U \, L \, U ^{-1}.
$$
Note that $U$ might also be written as $\left( -\L ^{\nabla^2 V} \right)^{-1/2} \nabla$. As a result, their spectra coincide. Summarizing our situation, we have
\begin{equation}
\label{eq:spectr_johnsen}
\sigma \left( - L \vert _{const ^\perp} \right) = \sigma \left( - \L ^{\nabla^2 V} \vert _{\nabla_I} \right) = \sigma \left( - \L _A ^{M_A} \vert _{\nabla_A} \right).
\end{equation}
Now, under the notation and assumptions of Theorem \ref{theo:multidim}, we wonder if the equality can hold in \eqref{eq:lambda_d}. Assume that the measure $\mu$ is strictly log-concave, i.e. the potential $V$ is strictly convex. If the spectral gap $\lambda_1$ is an eigenvalue of (maximal) multiplicity $d$ and the corresponding $d$ eigenfunctions define some diffeomorphism $H \in \C^\infty (\R^d, \R^d)$, then we have
$$
- J_{\L H} ^T \, (J_H ^T )^{-1} = J_H ^T \, \lambda_1 \, I \, (J_H ^T)^{-1} = \lambda_1 \, I,
$$
or, in other words, $M_A = \lambda_1 \, I$ so that from \eqref{eq:spectr_johnsen} we obtain in terms of spectra,
$$
\sigma (- L ) \backslash \{ 0 \} = \left \{ \lambda + \lambda_1 : \lambda \in \sigma \left( - \L_A \vert _{\nabla_A} \right) \right \} ,
$$
and since $\lambda_1$ has multiplicity $d$, we get the formula,
$$
\lambda_{d+1} = \lambda_1 ^A \vert _{\nabla_A} + \lambda_1 ,
$$
which is nothing but the equality in \eqref{eq:lambda_d}. Hence we observe that such an equality might hold, at the price of some strong assumptions on the various quantities of interest. Note however that with this choice of matrix $A$, the spectral gap $\lambda_1 ^A \vert _{\nabla_A}$ seems difficult to estimate since it depends on the vector field $H$ whose coordinates are the eigenfunctions associated to $\lambda_1$, which are unknown in general, except in some very particular cases. To conclude this section, we mention that F. Barthe informed us recently \cite{barthe-klartag} that the structure of the eigenspace associated to the spectral gap in the strictly log-concave case might be understood in some particular but non-trivial situations. Indeed, in a joint work in progress with B. Klartag, they established the following nice result: if we know that the spectral gap is attained, then its multiplicity is maximal at least when the potential $V$ shares the symmetries of the hypercube, that is, it is unconditional and also invariant with respect to the coordinates permutations. Hence such an observation might be useful for future investigation in this direction.

\section{Examples}
\label{sect:ex}
In this part, we concentrate on some situations where we can apply the spectral estimate of Theorem \ref{theo:multidim}.

\subsection{The general strategy}
\label{sec:strat}

Recall first that we need some invertible matrix of the type $A = (J_H ^T)^{-1}$, where $H\in \C^\infty (\R^d,\R^d)$ is some diffeomorphism such that on the one hand the matrix
$$
\nabla^2 V - \L A^{-1} \, A = - J _{\L H }^T \, (J_H ^T)^{-1},
$$
is symmetric and uniformly bounded from below by some positive constant and on the other hand that there exists a spectral gap for the restricted operator $-\L_A \vert _{\nabla_A}$. As we have seen previously, we have trivially $\lambda_1 ^A \vert _{\nabla_A} \geq \lambda_1 ^A $, a quantity which might be easier to estimate, and thus from now on we focus our attention on the latter spectral gap $\lambda_1 ^A $. \smallskip

\noindent Actually, let us assume in the sequel that $S$ is a diagonal matrix, since $\lambda_1 ^A $ seems to be inaccessible otherwise. We know then that $\lambda_1 ^A$ is the optimal constant $\lambda>0$ in the devoted Poincar\'e inequality: for every vector field $F \in \C ^\infty _0 (\R^d, \R^d)$,
$$
\lambda \, \sum_{i=1} ^d \int_{\R^d} \widetilde{F}_i ^2 \, S_{i,i} \, d\mu \leq \sum_{i=1} ^d  \int_{\R^d} \vert \nabla \widetilde{F}_i \vert ^2 \, S_{i,i} \, d\mu ,
$$
with
$$ \widetilde{F}_i = F_i - m_S (F) _i = F_i - \left( \int_{\R^d} S_{i,i} \, d\mu\right) ^{-1} \, \int_{\R^d} F_i \, S_{i,i} \, d\mu , \quad i = 1, \ldots, d .
$$
Hence we deduce the equality:
\begin{equation}
\label{eq:min}
\lambda_1 ^A = \min_{i = 1, \ldots, d} \lambda_1 ^i,
\end{equation}
where $\lambda_1 ^i$ is the spectral gap associated to the probability measure $\mu_A ^i$ on $\R^d$ with Lebesgue density proportional to $S_{i,i} \, e^{-V} $. Indeed, applying the Poincar\'e inequality for each measure $\mu_A^i$ directly gives the inequality $\geq$ in \eqref{eq:min}. The reverse inequality is obtained by considering vector fields $F$ with all coordinates vanishing except one. In the sequel, we denote $V_A ^i$ the potential associated to the measure $\mu_A ^i$, that is,
$$
V_A ^i := V - \log S_{i,i} ,
$$
and $L_A ^i$ the corresponding operator given for every $f\in \C ^\infty (\R^d, \R)$ by
$$
L_A ^i f = \Delta f- (\nabla V_A ^i) ^T \, \nabla f = L f + (\nabla \log S_{i,i})^T \, \nabla f .
$$
%D'ailleurs, il se peut que le trou spectral $\lambda_1 ^A$ vu comme un nombre ne soit pas la bonne quantite a considerer. Par exemple nous pourrions demander a ce que $\lambda_1 ^A$ soit une matrice symetrique constante commutant avec $S$, sachant que dans le cas d'un nombre c'est la matrice $\lambda_1 ^A \, I$ qui nous interesse. Toujours lorsque $S$ est diagonale, nous pouvons prendre alors $\lambda_1 ^A = \rm{diag} \lambda_1 ^i $ avec les memes $\lambda_1 ^i$ que precedemment. Cependant la definition de $\lambda_1 ^A$ par la formule variationnelle ne serait plus valable, ce qui n'est sans doute pas si important (l'inegalite de Poincare \eqref{eq:poincare_A} le serait elle, avec la matrice constante $\lambda_1 ^A$ dans l'integrale). \\
In terms of the diffeomorphism $H$, the matrix $S$ rewrites as
$$
S = J_H \, J_H ^T = (\nabla H_i \, \nabla H_j)_{i,j = 1, \ldots, d} ,
$$
and requiring that $S$ is a diagonal matrix means that the vector fields $(\nabla H_i)_{i=1,\ldots,d}$ form an orthogonal basis of $\R^d$ at each point $x\in \R^d$. When the jacobian matrix itself is diagonal, the diffeomorphism $H$ has its coordinates depending only on the $i^{th}$ coordinate $x_i$: $H_i (x) = h_i (x_i)$. The matrix $A$ is thus diagonal and is given by $A (x) = \mbox{diag} \, 1/h_i ' (x_i)$, $x\in \R^d$, where for a given vector $c \in \R^d$ we denote $\mbox{diag} \, c_i$ the diagonal matrix with $c_i$ on the $i^{th}$ line. As we will see below, this choice has the advantage to involve then some practical computations since on the one hand the matrix $- J _{\L H} ^T \, (J_H ^T)^{-1}$ is automatically symmetric and on the other hand the operator $\L _A$ is now a diagonal operator with the operator $L_A ^i$ on the $i^{th}$ line: for every $F\in \C ^\infty (\R^d,\R^d)$,
\begin{eqnarray*}
(\L_A F)_i (x) & = & L_A ^i F_i (x) \\
& = & L F_i (x) + \nabla \log (a_{i,i} ^{-2}) ^T \, \nabla F_i (x) \\
& = & L F_i (x) + 2 \, a_{i,i} (x) \, (\nabla a_{i,i} ^{-1})^T (x) \, \nabla F_i (x) \\
& = & L F_i (x) + 2 \, \frac{h_i '' (x_i)}{h_i ' (x_i)} \, \partial_i F_i (x).
\end{eqnarray*}
In particular, we have
\begin{eqnarray*}
- J _{\L H} ^T (x) \, (J_H ^T)^{-1} (x) & = & \left( - \frac{\partial_i L h_j (x)}{h_j '(x_j)}\right) _{i,j = 1,\ldots,d} \\
& = & \widetilde{\nabla^2 V} (x) - \mbox{diag} \, \frac{\partial_i L h_i (x)}{h_i '(x_i)} .
\end{eqnarray*}
The presence above of the null diagonal matrix $\widetilde{\nabla^2 V} := \nabla^2 V - \mbox{diag} \, \partial _{i,i} ^2 V$ suggests that we should not be far from Helffer's approach \cite{helffer} for estimating the spectral gap in some models arising in statistical mechanics, which focuses mainly on a uniform spectral gap assumption of the one-dimensional conditional distributions. See also the works of \cite{ledoux,gr,chen} and more recently \cite{barthe-cordero} in which the principle of the method is nicely and shortly summarized. However it reveals that our intertwining method is different from that emphasized by Helffer since ours is more global in space and avoids the use of these one-dimensional conditional distributions.

\subsection{The case of perturbed product measures}
In order to have a better comprehension of how our criteria involved in Theorem \ref{theo:multidim} could be applied, let us consider the case of a perturbed product measure, namely the potential $V$ is given by
$$
V (x) = \sum_{i=1} ^d U_i (x_i) + \varphi (x), \quad x\in \R ^d ,
$$
and is assumed to be sufficiently smooth on $\R^d$ and such that the associated measure $\mu$ is a probability measure. If the potentials $U_i$ can be written at least as the sum of uniformly convex and bounded potentials on $\R$, and the function $\varphi$ is convex on $\R^d$ (the latter assumption might be weakened to a sufficiently small non-positive lower bound on its Hessian matrix), then the whole potential $V$ itself can be decomposed as the sum of uniformly convex and bounded potentials on $\R^d$ so that the perturbation principle of Holley and Stroock applies, cf. \cite{holley-stroock}, the lower bound obtained on the spectral gap depending poorly on the dimension. We will see below that the intertwining approach leads us to a different set of assumptions, allowing us to consider some interesting cases which do not enter into the previous framework. \smallskip

Although the forthcoming analysis might be adapted to the non-convex case, let us assume for the sake of simplicity that the $U_i$ are convex. Since we have $\widetilde{\nabla^2 V} = \widetilde{\nabla^2 \varphi}$, some computations give us
\begin{eqnarray*}
- J _{\L H} ^T (x) \, (J_H ^T)^{-1} (x) & = & \widetilde{\nabla^2 V} (x) - \mbox{diag} \, \frac{\partial_i L h_i (x)}{h_i '(x_i)} \\
& = & \widetilde{\nabla^2 \varphi} (x) + \mbox{diag} \, \frac{(- L_i h_i)' (x_i)}{h_i ' (x_i)} + \mbox{diag} \, \partial^2 _{i,i} \varphi (x) \\
& & + \frac{\partial_i \varphi (x) \, h_i '' (x_i)}{h_i ' (x_i)},
\end{eqnarray*}
where $L_i$ denotes the one-dimensional dynamics defined as
$$
L_i h(y) = h'' (y) - U_i ' (y) \, h ' (y) , \quad y\in \R,
$$
having an invariant measure whose Lebesgue density on $\R$ is proportional to $e^{-U_i}$. Choosing then the one-dimensional functions $h_i ' = e^{\varepsilon _i U_i}$ with the parameters $\varepsilon _i \in (0,1/2)$ (so that $h_i ' \in L^2 (\mu)$) to be determined according to a case-by-case examination, we obtain
$$
- J _{\L H} ^T (x) \, (J_H ^T)^{-1} (x) = \nabla^2 \varphi (x) + \mbox{diag} \, \left( (1-\varepsilon_i) \left( U_i '' (x_i) + \varepsilon_i U_i ' (x_i) ^2 \right) + \varepsilon _i \partial_i \varphi (x) U_i ' (x_i) \right).
$$

To ensure that the above matrix is uniformly bounded from below by some positive constant, the trivial inequality $ab \geq - a^2/2 - b^2/2$, $a, b\in \R$, applied to $a= \partial _i \varphi (x)$ and $b=\varepsilon _i \, U_i ' (x_i)$ leads us to the following assumption: for every $i\in \{ 1,\ldots, d\}$ there exists $\alpha_i >0$ such that for every $x\in \R^d$,
\begin{equation}
\label{eq:U_i}
\rho (\nabla^2 \varphi)(x) - \frac{\partial _i \varphi (x) ^2}{2} \,  + (1-\varepsilon_i) U_i ''(x_i) + \varepsilon_i (1 - 3 \varepsilon_i /2) U_i ' (x_i) ^2 \geq \alpha_i .
\end{equation}
Therefore Theorem \ref{theo:multidim} applies once we estimate from below the spectral gap $\lambda_1 ^A$, which reduces by \eqref{eq:min} to estimate the spectral gap $\lambda_1 ^i$ with a lower bound uniform in $i \in \{ 1,\ldots, d\}$. In the present context we have for every $i\in \{ 1,\ldots, d\}$,
$$
e^{-V_A ^i (x)} = S_{i,i} (x) \, e^{-V(x)} = h_i ' (x_i) ^2 \, \exp \left( - \sum_{j=1} ^d U_j (x_j) - \varphi (x) \right) ,
$$
so that the potential $V_A ^i$ rewrites as
$$
V_A ^i (x) = \sum_{j \neq i}  U_j (x_j) + (1- 2 \varepsilon_i) \, U_i (x_i) + \varphi (x), \quad x\in \R ^d.
$$
Therefore the spectral gap $\lambda_1 ^i$ might be estimated from below by adapting immediately the inequality \eqref{eq:gap}. More precisely, we have to find some invertible matrix $B$ such that the matrix
$$
\nabla^2 V_A ^i - \L _A ^i B^{-1} \, B,
$$
is symmetric and uniformly bounded from below by some positive constant, and in this case we would obtain the estimate
$$
\lambda_1 ^i \geq \inf \, \rho \left( \nabla^2 V_A ^i - \L _A ^i B^{-1} \, B \right) .
$$
Above the matrix operator $\L _A ^i $ is given by
\[
\L_A ^i = \left(
\begin{array}{ccc}
L_A ^i & & \\
& \ddots & \\
& & L_A ^i
\end{array}
\right) .
\]
Since the measures $\mu$ and $\mu_A ^i$ are somewhat similar, the one-dimensional potential $U_i$ being replaced by $(1-2\varepsilon _i)U_i$, a convenient choice of matrix $B$ is to take a small variation of $A$, i.e. diagonal with one-dimensional functions of the form $e^{-\varepsilon_j U_j}$ for $j\neq i$, and $e^{-\varepsilon_i (1- 2 \varepsilon_i) U_i}$ on the $i^{th}$ line. In other words, the two matrices $A$ and $B$ coincide except on the $i^{th}$ line. Hence, we assume that there exists $\beta_i>0$ such that for every $x\in \R^d$,
$$
\rho (\nabla^2 \varphi)(x) - \frac{\partial _i \varphi (x) ^2}{2} \,  + (1-\varepsilon_i)(1-2\varepsilon_i) U_i ''(x_i) + \varepsilon_i (1 - 3 \varepsilon_i /2) (1-2\varepsilon_i)^2 U_i ' (x_i) ^2 \geq \beta_i .
$$
Note that since $U_i$ is convex and $\varepsilon_i \in (0,1/2)$, we expect in practice $\beta_i \leq \alpha_i$. Finally, we obtain the lower estimate
$$
\lambda_1 ^i \geq \min \{ \min _{j \neq i } \alpha_j , \beta_i \},
$$
and by the spectral estimate of Theorem \ref{theo:multidim}, we get
\begin{eqnarray*}
\lambda_{d+1} & \geq & \min _{i=1,\ldots,d} \alpha_i + \min _{i=1,\ldots,d} \min \{ \min _{j \neq i } \alpha_j , \beta_i \} .
\end{eqnarray*}
Let us summarize our result, including a slight modification of the assumptions above, to allow an easier use for applications.
\begin{prop}\label{prop:exemple}
With the above notation, assume that the one-dimensional potentials $U_i$ are convex and that there exist constants $c_1 \in \R$, $c_2 >0$ and $\gamma >0$ such that:
\begin{itemize}
\item $\inf \, \rho (\nabla^2 \varphi) \geq c_1$ and $\underset{i = 1,\ldots, d}{\max} \, \Vert \partial_i \varphi \Vert _\infty \leq c_2$, \smallskip
\item for every $i\in \{ 1,\ldots, d \}$ there exists $\ve_i\in (0,1/2)$ such that:
$$
(1- 3\varepsilon_i /2) \, (1-2\varepsilon_i) ^2 \, \left( U_i ''(x_i) + \varepsilon_i \, U_i ' (x_i) ^2 \right) \geq \gamma.
$$
\end{itemize}
Then, we obtain the following eigenvalue lower bounds:
\[
\lambda_1 \geq \gamma + c_1- \frac{c_2^2}{2} \quad \mbox{ and } \quad \lambda_{d+1} \geq 2 \gamma +c_1- \frac{c_2^2}{2}.
\]
\end{prop}

Certainly our estimates above, which are meaningful only if the lower bounds are positive, have no reason to be sharp, meaning that there is still room for improvement in the choice of the (possibly non-diagonal) matrix $A$. However they offer a robust and rather easy way to derive lower estimates on these spectral quantities of interest which are not necessarily true eigenvalues (they might correspond to the bottom of the essential spectrum). \smallskip

To see the relevance of our approach, let us consider the one-dimensional potentials $U_i (x_i) = \vert x_i \vert ^a / a$ with $a \in (1,2)$, and a potential of interaction of the type
$$
\varphi(x)= c \, \sum_{i=1}^{d} \vert x_{i+1} -x_i\vert , \quad x\in \R^d,
$$
with the convention $x_{d+1} := x_1$ and where $c>0$ is a sufficiently small constant. Note first that there are two main difficulties when trying to use Helffer's classical approach for this example: on the one hand dealing with the one-dimensional potential $U_i$, it does not satisfy the uniform spectral gap assumption required by the method, as noticed in \cite{gr} (the reason is that the infimum of $U_i ''$ is 0 and is attained at infinity) and on the other hand the strong condition $\max_{ i\in \{ 1,\ldots,d \}} \, \sup_{x \in \R^d} \partial _{i,i} ^2 \varphi(x) < \infty$ required in \cite{ledoux} on the potential of interaction is not satisfied. More precisely its regularized version
$$
\varphi _\tau (x)= c\, \sum_{i=1}^{d} \sqrt{\tau ^2 + (x_{i+1} -x_i )^2 }, \quad \mbox{ with } \quad x_{d+1} := x_1,
$$
does not satisfy this assumption as $\tau \to 0$. However we are able to apply Proposition \ref{prop:exemple} as follows. After some computations on the regularized version and taking the limit $\tau \to 0$, we obtain $c_1 = 0$ and $c_2 = 2 c$. To compute the parameter $\gamma$, we have first
$$
\inf_{x_i \in \R } \, U_i ''(x_i) + \varepsilon_i \, U_i ' (x_i) ^2 = (a-1) \, \left( \frac{2-a}{2\varepsilon_i} \right) ^{1-2/a} + \varepsilon_i \, \left( \frac{2-a}{2\varepsilon_i} \right) ^{2-2/a}.
$$
Choosing then all the $\varepsilon_i := 1-a/2 \in (0,1/2)$ leads us to the constant $\gamma = (3a-2)(a-1)^2 a /8$, so that we obtain finally the estimates:
$$
\lambda_1 \geq \frac{(3a-2) \, (a-1)^2 \, a}{8} - 2c^2 \quad \mbox{ and } \quad \lambda_{d+1} \geq \frac{(3a-2) \, (a-1)^2 \, a}{4} - 2c^2 .
$$

\end{document}